\documentclass[12pt]{article}
\usepackage{tikz}
\usepackage{amsthm,amssymb}

\newtheorem{theorem}{Theorem}

\newtheorem{lemma}[theorem]{Lemma}

\def\cal{\mathcal}

\def\floor#1{\lfloor#1\rfloor}

\def\R{{\mathbb R}}

\date{\today}

\title{The motif problem}

\author{
E. Rodney Canfield\\
\small Department of Computer Science\\[-0.8ex]
\small University of Georgia\\[-0.8ex]
\small\tt erc@cs.uga.edu\\
\and
Ron Fertig\\
\small Center for Communications Research\\[-0.8ex] 
\small San Diego, CA\\[-0.8ex]
\small\tt rfertig@ccrwest.org\\
\and
R. Daniel Mauldin\\
\small Department of Mathematics\\[-0.8ex] 
\small The University of Texas\\[-0.8ex]
\small\tt rdmauldin@unt.edu\\
\and David Moews\\
\small Center for Communications Research\\[-0.8ex] 
\small San Diego, CA\\[-0.8ex]
\small\tt dmoews@ccrwest.org\\}

\begin{document}

\maketitle
\begin{abstract}  Fix a choice and ordering of four pairwise non-adjacent
vertices of a parallelepiped, and call a {\em motif} a sequence of four
points in $\R^3$ that coincide with these vertices for some, possibly
degenerate, parallelepiped whose edges are parallel to the axes.
We show that a set of $r$ points can contain at most $r^2$ motifs.
Generalizing the notion of motif to a sequence of $L$ points in $\R^p$,
we show that the maximum number of motifs that can occur in a point set
of a given size is related to a linear programming problem arising from
hypergraph theory, and discuss some related questions.
\end{abstract}
\renewcommand\thefootnote\relax

\section{Introduction}\label{intro}

\begin{figure}
\begin{center}

$$
\phantom{\kern3.5in}\begin{array}{c|ccc}
&x&y&z\\
a&q&u&v\\
b&s&t&v\\
c&q&t&w\\
d&s&u&w
\end{array}
$$
\\

\vskip -1.4in
\begin{tikzpicture}[scale=1.5]
\begin{scope}
\draw (0,0) -- (1,0) -- (0.5,-1.5) -- (-0.5,-1.5) -- (0,0) [line width=0.02cm];
\draw (0,-3) -- (1,-3) -- (0.5,-4.5) -- (-0.5,-4.5) -- (0,-3) [line width=0.02cm];
\draw (0,0) -- (0,-3) [line width=0.02cm];
\draw (1,0) -- (1,-3) [line width=0.02cm];
\draw (-0.5,-1.5) -- (-0.5,-4.5) [line width=0.02cm];
\draw (0.5,-1.5) -- (0.5,-4.5) [line width=0.02cm];
\draw (-0.5,0) node [anchor=west] {a};
\draw (0.5,-1.5) node [anchor=west] {b};
\draw (1.0,-3) node [anchor=west] {c};
\draw (-1.0,-4.5) node [anchor=west] {d};

\draw[->] (2,-1.5) -- (2.3,-1.5) [line width=0.02cm];
\draw (2.3,-1.5) node [anchor=west] {y};
\draw[->] (2,-1.5) -- (2,-1.2) [line width=0.02cm];
\draw (2,-1.2) node [anchor=south] {z};
\draw[->] (2,-1.5) -- (1.9,-1.8) [line width=0.02cm];
\draw (1.9,-1.8) node [anchor=east] {x};

\end{scope}
\end{tikzpicture}\\
Figure 1. A  motif.
\end{center} 
\end{figure}

Let a {\em motif} be a quadruple of four points $(a,b,c,d)$ in $\R^3$
that lie on pairwise non-adjacent 
vertices of a parallelepiped with edges parallel to the
$x$-, $y$-, and $z$-axes, as shown in Figure 1.  
We allow the
parallelepiped to be a reflection of the one shown here, or to be degenerate,
so that the only constraint on a motif is that there are certain equality
constraints between the coordinates of $a$, $b$, $c$, and $d$; 
we require the $x$ coordinates of $a$ and $c$ to be the same,
the $y$ coordinates of $b$ and $c$ to be the same, and so on,
as shown in the table in Figure 1.  (The term is not to be confused
with the notion of motif as used in algebraic geometry.)

Given a set $\cal S$ of $r$ distinct points in $\R^3$, how many motifs can be 
made out of the points in $\cal S$?  We will show that there can be at most
$r^2$.  Foliate $\R^3$ by planes parallel to the $xy$-plane, and suppose
that there are $n$ such planes, $P_1$, \dots, $P_n$, say, that  contain points 
of $\cal S$, containing $r_1$, \dots, $r_n$ points of 
$\cal S$, respectively.  We may order the planes so that 
$r_1\ge \cdots\ge r_n>0$.  Now, in any motif, $a$ and $b$ must be contained
in the same plane, $P_i$ say, and $c$ and $d$ must be in the same plane, $P_j$,
say.  If $i\le j$, then fix $j$, and observe that $a$ and $b$ are determined
by $c$, $d$, and $i$; there are then at most $j$ ways of choosing $i$,
and $r_j^2$ ways of choosing $c$ and $d$, so our choice of $j$ contributes
at most $j r_j^2$ motifs.  On the other hand, if $i>j$, fix $i$, and observe
that $c$ and $d$ are determined by $a$, $b$, and $j$, and that there
are at most $i-1$ ways of choosing $j$, so our choice of $i$ contributes
at most $(i-1) r_i^2$ motifs.  Therefore, the total 
number of motifs possible is no more than
\begin{eqnarray*}
\sum_{1\le j\le n} j r_j^2 + \sum_{2\le i\le n} (i-1) r_i^2
&=& \sum_{1\le i\le n} (2i-1) r_i^2 \\
&\le& \sum_{1\le i\le n} r_i(2r_1+\cdots+2r_{i-1}+r_i)\\
&=& (r_1+\cdots+r_n)^2=r^2, \qquad \hbox{as desired.}
\end{eqnarray*}

On the other hand, if we choose $\cal S$ to be all points in a 
three-dimensional grid of $A$ points by $B$ points by $C$ points, 
$\cal S$ will have size $r=ABC$ and will admit $A^2B^2C^2=r^2$ motifs.  This 
shows that the upper bound $r^2$ is optimal.

\section{A generalization}

We generalize the notion of motif to $L$-tuples of points in $\R^p$
as follows: given positive integers
$L$ and $p$, call a {\em motif specification on $p,L$} a sequence $\pi_1$,
\dots, $\pi_p$ of partitions of subsets of $\{1,\ldots,L\}$, together
with constants $D_{im}\in \R$ for each pair $(i,m)$ such that
$i\in\{1,\ldots,L\}$, $m\in\{1,\ldots,p\}$, and $i$ is in no block of $\pi_m$.
A tuple $(f_1,\ldots,f_L)\in (\R^p)^L$ will then be a {\em motif} (for
that specification) if:
\begin{enumerate}
\item Whenever $i$, $j\in\{1,\ldots,L\}$ are in the same block of
$\pi_m$, for some $m\in\{1,\ldots,p\}$, then $(f_i)_m=(f_j)_m$.
\item Whenever $i\in\{1,\ldots,L\}$ is in no block of $\pi_m$,
for some $m\in\{1,\ldots,p\}$, then $(f_i)_m=D_{im}$.
\end{enumerate}

Given a subset ${\cal S}\subseteq \R^p$ and a motif specification, we 
let ${\cal M}({\cal S})$ be the number of motifs in $\cal S$,
i.e., the number of $L$-tuples of elements of $\cal S$ that  are motifs.
We want to know how large ${\cal M}({\cal S})$ can be 
for a set $\cal S$ of given size.  To state and prove an estimate
for this, we make the following definitions:
\begin{itemize}
\item Let the {\em hypergraph of a motif specification} be
the hypergraph with vertex set $\{1,\ldots,L\}$ and an edge
for each block of each partition $\pi_1$, \dots, $\pi_p$, 
such that each edge is incident on the vertices that  are elements
of the corresponding block.
\item Given a hypergraph $H$, let $V(H)$ be the set of its vertices
and $E(H)$ be the set of its edges.  We write $x\in e$ if the edge $e$
is incident on the vertex $x$.
\item \cite{furedi} A hypergraph is {\em uniform} if all its edges have the same size.
\item \cite{lovasz}\cite{furedi} A {\em fractional transversal} $g$ of a hypergraph $H$ is
a function $g: V(H)\rightarrow{\mathbb R}_{\ge 0}$ such that
$\sum_{x \in e} g(x) \ge  1$ for all $e \in E(H)$.
We call $\sum_x g(x)$ the {\em total weight} of $g$.
\item \cite{lovasz}\cite{furedi} The {\em fractional transversal number}, $\tau^*(H)$, of a
hypergraph $H$ is defined to be
$$
\tau^*(H) := \min \{ \sum_x g(x)\  \mid \hbox{$g$ a fractional transversal of $H$.}\}
$$
\item \cite{lovasz}\cite{furedi} A {\em fractional matching} $f$ of a hypergraph $H$ is
a function $f: E(H)\rightarrow{\mathbb R}_{\ge 0}$ such that
$\sum_{x \in e} f(e) \le  1$ for all $x \in V(H)$.
We call $\sum_e f(e)$ the {\em total weight} of $f$.
\item \cite{lovasz}\cite{furedi} The {\em fractional matching number}, $\nu^*(H)$, of a
hypergraph $H$ is defined to be
$$
\nu^*(H) := \max \{ \sum_e f(e)\  \mid \hbox{$f$ a fractional matching of $H$.}\}
$$
\item Given any $m>0$ and any function $F: \R^m\rightarrow \R$, write *
for an argument of $F$ in order to denote summing over that argument:
\begin{eqnarray*}
G(*)&=&\sum_{a\in\R} G(a),\\
F(*,*,c)&=&\sum_{a,b\in\R} F(a,b,c),\\
\hbox{etc.}&&
\end{eqnarray*}
\end{itemize}

The hypergraph of a motif specification may contain multiple edges which
are incident on the same set of vertices.  However, this need not be 
considered when computing $\tau^*$ and $\nu^*$, since if the hypergraph $H'$ is 
obtained from $H$ by omitting duplicate edges, then $\tau^*(H')=\tau^*(H)$ and 
$\nu^*(H')=\nu^*(H)$.

\begin{theorem}  \cite{lovasz}\cite{furedi} For all hypergraphs $H$, $\tau^*(H)=\nu^*(H)$.  Also, if a fractional transversal and fractional matching of $H$
have the same total weight $w$, then $w=\tau^*(H)=\nu^*(H)$.
\end{theorem}
\begin{proof}
If $M$ is the incidence matrix of $H$, $\tau^*(H)$ is the optimal value for
the linear programming problem of minimizing 
${\bf 1}^T {\bf x}$, where ${\bf x}\ge 0$ and $M {\bf x}\ge {\bf 1}$.
Also, $\nu^*(H)$ is the optimal value for the linear programming problem
of maximizing ${\bf 1}^T {\bf y}$, where ${\bf y}\ge 0$ and 
$M^T {\bf y}\le {\bf 1}$.  However, these problems are dual, so they have the 
same optimal value and $\tau^*(H)=\nu^*(H)$.  The second claim now follows 
from the definitions of $\tau^*$ and $\nu^*$.
\end{proof}

In the case of our introductory example, $L=4$ and $p=3$,
and the motif specification uses partitions 
$\pi_1=\{\{1,3\},\{2,4\}\}$, $\pi_2=\{\{1,4\},\{2,3\}\}$,
and $\pi_3=\{\{1,2\},\{3,4\}\}$.  The hypergraph $H$ is therefore
a complete graph on 4 vertices, and one fractional transversal of $H$
is the function $g_0$ that  takes all vertices to $1/2$.  Similarly,
one fractional matching of $H$ is the function $f_0$ that  takes all
edges to $1/3$.  Since these have the same total weight, 2, we see that
$\tau^*(H)=\nu^*(H)=2$, and that $g_0$ and $f_0$ are optimal.

We can now prove an upper bound on ${\cal M}({\cal S})$ using 
the next lemma.  Take $0^0$ to be 1.

\begin{lemma}\label{l1} 
Given some $n>0$, 
functions $F_1$, \dots, $F_n: \R\rightarrow \R_{\ge 0}$
each of which is nonzero at most at only finitely many places,
and exponents $h_1$, \dots, $h_n\ge 0$ such that
$h_1+\cdots+h_n\ge 1$, then
$$
\sum_{a\in\R} F_1(a)^{h_1}\cdots F_n(a)^{h_n}
\le F_1(*)^{h_1}\cdots F_n(*)^{h_n}.
$$
\end{lemma}
\begin{proof}
Zero exponents do not contribute to the product on either side, so discard
them, and assume that $h_i>0$ for all $i$.  Then, in the case where $\sum_i h_i=1$, the inequality is Theorem 11 of \cite{hwp}, which is a form of H\"older's 
inequality.  If $\sum_i h_i>1$, the inequality is Theorem 22 of \cite{hwp}, and follows from the $\sum_i h_i=1$ case together with the inequality
$$
\left(\sum_{\alpha} x_{\alpha}^r\right)^{1/r}
\le
\sum_{\alpha} x_{\alpha},
\qquad r>1, x_{\alpha}\ge 0.
$$
\end{proof}

\begin{theorem}  
\label{thm1}
Given any motif specification on $p,L$, let $H$ be its hypergraph,
and let $\cal S$ be a subset of $\R^p$ with size $r$.  Then 
${\cal M}({\cal S})\le r^{\tau^*(H)}$.
\end{theorem}
\begin{proof} 
Let $I:\R^p\rightarrow\{0,1\}$ be the indicator function of $\cal S$.
Fix a motif specification.  Any motif can then be written as an $L$ by $p$
matrix of coordinates 
whose entries are variables and constants, where there is one variable for 
each block of each $\pi_i$, and the constants are the $D_{im}$'s, placed
in the appropriate spots of the matrix.  For instance, 
for the motif specification
used in the introduction, the matrix will be a 4 by 3 matrix 
containing six variables, as shown on the right-hand side of Figure 1.
In general, we may suppose that we use variables $v_1$, \dots, $v_n$,
and we can then write the number of motifs in $\cal S$ as
\begin{equation}
\label{eone}
{\cal M}({\cal S})=\sum_{v_1,\ldots,v_n\in \R} I(R_1) \cdots I(R_L),
\end{equation}
where $R_1$, \dots, $R_L$ are $p$-tuples of variables and constants obtained
from reading off the rows of the matrix.
Now, taking the number 0 or 1 and raising it to an arbitrary nonnegative
power will not decrease it.  So, from
(\ref{eone}), we find that the number of motifs in $\cal S$ is no more than
\begin{equation}
\label{etwo}
Y=\sum_{v_1,\ldots,v_n\in \R} I(R_1)^{g(1)} \cdots I(R_L)^{g(L)},
\end{equation}
where $g$ is an optimal fractional transversal of $H$.
For example, using the matrix in Figure 1 and the
optimal fractional transversal $g_0$ mentioned above, (\ref{etwo})
will take the form
$$
Y=\sum_{q,s,t,u,v,w\in \R} I(q,u,v)^{1/2} I(s,t,v)^{1/2} I(q,t,w)^{1/2} I(s,u,w)^{1/2}.
$$

Now, observe that any given variable in the sum (\ref{etwo}) 
corresponds to an edge, $e$, say, 
of $H$, and that the sum of the exponents of the $I(R_i)$'s in which this 
variable appears equals $\sum_{x\in e} g(x)$, which, by the definition of
a fractional transversal, is at least 1.  We can therefore apply Lemma 
\ref{l1} to eliminate the sum over this variable, replacing it by * wherever 
it appears.  This will not decrease the value of (\ref{etwo}).  In the case 
given in Figure 1, for example, we observe that
$$
\sum_q I(q,u,v)^{1/2} I(q,t,w)^{1/2}\le I(*,u,v)^{1/2} I(*,t,w)^{1/2}
$$
and so 
$$
Y\le\sum_{s,t,u,v,w\in \R} I(*,u,v)^{1/2} I(s,t,v)^{1/2} I(*,t,w)^{1/2} I(s,u,w)^{1/2}.
$$
If we repeat this operation until all variables are removed, we will find that 
$$Y\le I(R'_1)^{g(1)} \cdots I(R'_L)^{g(L)},$$
where each $R'_i$ is a $p$-tuple each of whose elements is either a constant
or *.  Since each $I(R'_i)$ can be no bigger than $r$, this shows that
$\cal S$ has no more than $r^{\tau^*(H)}$ motifs, as desired.
\end{proof}

We can also prove a lower bound which shows that in infinitely many
cases, this upper bound is best possible, up to a constant factor.

\begin{theorem} \label{thm2}
Given any motif specification on $p,L$, let $H$ be its hypergraph.
Then, for all $N>0$, there exists a subset $\cal S$ of $\R^p$ such that
the size, $r$, of $\cal S$ is at least $N$, and 
${\cal M}({\cal S})\ge (r/L)^{\nu^*(H)}$.
\end{theorem}
\begin{proof}
Let $f$ be an optimal fractional matching of $H$ all of whose values
are rational.  If $f(e)=0$ for
all edges $e$ of $H$, then $\nu^*(H)=0$ and the theorem is trivial,
since it reduces to saying that there exist arbitrarily large sets
with at least one motif.  Otherwise, $f$ has at least one positive
value.  Let $d>0$ be a common multiple of the denominators of all positive
values of $f$, and write down the $L$ by $p$ matrix of coordinates,
as we did in Theorem \ref{thm1}.  Now, fix a positive integer $M$,
and let each variable $v$ in the matrix of coordinates run over
$M^{d f(e)}$ values, where $e$ is the edge of $H$ corresponding to $v$.
The number of values taken on by the row of the matrix corresponding to the 
vertex $x$ will then be $M^{d \sum_{x\in e} f(e)}$, which, by the definition
of a fractional matching, is at most $M^d$.  Therefore, taking the set
of all possible rows of the matrix will construct a set $\cal S$
with size at most $L M^d$.  Also, we must have $\sum_{x \in e} f(e)=1$
for at least one vertex $x$, since if not, we could increase any
$f(e)$ slightly and still have a fractional matching, contradicting
optimality of $f$.  It follows that the size of $\cal S$ must be at least
$M^d$.  Finally, the number of motifs in $\cal S$ is at least
$M^{d \sum_e f(e)}=M^{d \nu^*(H)}$.  Letting $M$ become large completes
the proof.
\end{proof}

{\bf Remark.} Start with the matrix of coordinates used in Theorems \ref{thm1} and \ref{thm2},
and make a tableau by leaving the constants in the matrix unchanged, 
but replacing 
each variable $v$ with the indeterminate $X_{f(e)}$, where $e$ is the edge of 
$H$ corresponding to $v$, and $f$ is some optimal fractional matching of $H$.  
Then it follows from the proof of Theorem \ref{thm2} that Theorem \ref{thm2} 
may be improved by replacing $L$ with $L'$,
where $L'$ is the number of distinct rows in the tableau.

\section{Uniform motifs and grids}

Suppose that the motif specification contains no constants and that its
hypergraph is uniform, i.e., all blocks in all partitions have the same size,
$n$, say.  We may then reorder the coordinates of $\R^p$ so as to place
each set of identical partitions in the specification into a contiguous range; 
let there be $q$ such ranges,
with lengths, in order, $p_1$, \dots, $p_q$ ($p_1+\cdots+p_q=p$.)

\begin{theorem} 
Let a motif specification on $p,L$ have no constants and a uniform hypergraph,
with edge size $n$, and let its coordinates be ordered as above.  
Then, for all $r\ge 0$, the maximum number of motifs in a subset $\cal S$ of 
$\R^p$ of size $r$
is always $r^{L/n}$, and this maximum is attained iff $\cal S$ is
a grid, i.e., $\cal S$ is of the form 
\begin{equation}
\label{gridform}
{\cal S}_1\times\cdots\times {\cal S}_q, \hbox{where ${\cal S}_i\subseteq
\R^{p_i}$ for each $i=1$, \dots, $q$.}
\end{equation}
\end{theorem}
\begin{proof}
If $\pi_i=\pi_j$, then coordinates $i$ and $j$ of $\R^p$ can be lumped
together and treated as a single coordinate.  Doing this repeatedly,
we reduce to the case where $q=p$ and $p_1=\cdots=p_q=1$.

If we let $g$ have value $1/n$ on each vertex, then $g$ is a fractional
transversal of the hypergraph $H$ of the motif specification.  Also,
if we let $f$ have value $1/p$ on each edge, then $f$ is a fractional
matching of $H$.  However, $f$ and $g$ both have weight $L/n$.  It follows 
that $\tau^*(H)=\nu^*(H)=L/n$.
Now, if we let $\cal S$ contain $r$ points which have all coordinates zero 
except the first, the first coordinate ranging over 1, \dots, $r$, then 
$\cal S$ will have exactly $r^{L/n}$ motifs.  Together with Theorem \ref{thm1},
this proves that the maximum number of motifs in a set of size $r$ is
$r^{L/n}$.  

If $\cal S$ is of the form (\ref{gridform}),  then write down
the $L$ by $p$ matrix of coordinates referred to in Theorem \ref{thm1}.  
We can obtain a motif in $\cal S$
by, for each $i$, substituting an arbitrarily chosen element of ${\cal S}_i$
for each variable in column $i$.  If each set ${\cal S}_i$ has
size $r_i$ ($i=1$, \dots, $p$), it follows that $\cal S$ contains at least
$$
\prod_{1\le i\le p} r_i^{L/n} = (\prod_{1\le i\le p} r_i)^{L/n} = r^{L/n}
$$
motifs.
It remains to prove that any set which attains the maximum number
of motifs for its size is of the form (\ref{gridform}).

Suppose that $\cal S$ is a set which attains the maximum number
of motifs for its size, and let 
$\cal S$ have indicator function $I:\R^p\rightarrow\R$.  We may
assume that $\cal S$ is nonempty, as otherwise the result is trivial.
Now, in
the proof of Theorem \ref{thm1}, we construct a chain of inequalities
starting from ${\cal M}({\cal S})$ and ending at $r^{L/n}$.
To have $r^{L/n}$ motifs in $\cal S$, all these inequalities must be
equalities.  Looking at the matrix of coordinates, we see
that a new inequality appears in the chain whenever we remove a
variable and replace it by *.  Since we can choose to remove the variables
in any order, it follows that we must have
\begin{equation}
\label{e98}
\sum_q I(R_1)^{1/n}\cdots I(R_L)^{1/n}
=
I(R'_1)^{1/n} \cdots I(R'_L)^{1/n},
\end{equation}
whenever $R_1$, \dots, $R_L$ result from taking the rows of the
matrix of coordinates and substituting each variable by either a real number, 
*, or the variable $q$, and the $R'_i$s are derived from the $R_i$s by 
substituting * for $q$.

We prove by induction on $s$ that for all $s=0$, \dots, $p$,
there is a constant $C>0$ (namely $C = I(*,\dots,*)$), and
functions $F_1$, \dots, $F_s:\R\rightarrow \R_{\ge 0}$,
each zero except at finitely many points, such that
\begin{equation}
\label{e99}
I(b_1,\ldots,b_s,*,\ldots,*)
=CF_1(b_1)\cdots F_s(b_s),\qquad \hbox{for all $b_1$, \dots, $b_s\in\R$.}
\end{equation}
If $s=0$, this is clear, with the given value of $C$, since there are
no functions $F_i$.  Otherwise, assume that (\ref{e99}) has been
proved for $s=t-1$.  We wish to prove it for $s=t$.  Fix some
$c_1$, \dots, $c_{t-1}$ such that $I(c_1,\ldots,c_{t-1},*,\ldots,*)>0$,
and set
$$
F_t(q):=\frac{I(c_1,\ldots,c_{t-1},q,*,\ldots,*)}
            {I(c_1,\ldots,c_{t-1},*,*,\ldots,*)},
\qquad \hbox{$q\in\R$.}
$$

We now induce secondarily on the number of coordinates $j$ in 
$\{1,\ldots,t-1\}$ for which $b_j\ne c_j$.
Let this number be $m$.  We have already proved (\ref{e99}) for $s=t$ and 
$m=0$.  Suppose that it has been proved for $s=t$ and 
$m=0$, \dots, $m'-1$.  We wish to prove
it for $s=t$ and
$m=m'$.  Fix some $b_1$, \dots, $b_{t-1}$ which differ from the
$c_i$'s in exactly $m'$ places.  Assume that $I(b_1,\ldots,b_{t-1},*,\ldots,*)>0$; if not, by (\ref{e99}) for $s=t-1$, at least one of $F_1(b_1)$,
\dots, $F_{t-1}(b_{t-1})$ must be zero, so both sides of (\ref{e99}) are
zero and we are done.  Pick some $i\in\{1,\ldots,t-1\}$ for which $b_i\ne c_i$,
and set $b'_j=b_j$ if $j\ne i$ and $j\in\{1,\ldots,t-1\}$, $b'_i=c_i$.
We now fill in the matrix of coordinates with four different
possibilities for a row:
\begin{eqnarray*}
R''_1 &=& (b_1, \ldots, b_{t-1}, q, *, \ldots, *), \\
R''_2 &=& (b'_1, \ldots, b'_{t-1}, q, *, \ldots, *), \\
R''_3 &=& (b_1, \ldots, b_{t-1}, *, *, \ldots, *), \ \ \hbox{and} \\
R''_4 &=& (b'_1, \ldots, b'_{t-1}, *, *, \ldots, *).
\end{eqnarray*}
The rows are determined as follows. By assumption, $\pi_i\ne \pi_t$, so we can find some $u$ and $v$
which are in the same block of $\pi_t$ but different
blocks of $\pi_i$.  We set $R_w:=R''_1$ for each row $w$ such that 
$w$ is in the same block of $\pi_t$ as $u$ and $v$, and $w$ is in the
same block of $\pi_i$ as $u$; $R_w:=R''_2$ for each row $w$ such that
$w$ is in the same block of $\pi_t$ as $u$ and $v$, but $w$ is not in the
same block of $\pi_i$ as $u$;
$R_w:=R''_3$ for each row $w$ such that
$w$ is not in the same block of $\pi_t$ as $u$ and $v$, but $w$ is in the
same block of $\pi_i$ as $u$;
and $R_w:=R''_4$ in rows $w$ such that 
$w$ is not in the same block of $\pi_t$ as $u$ and $v$, and $w$ is not in the
same block of $\pi_i$ as $u$.
We then apply (\ref{e98}).  We observe that by assumption
and by (\ref{e99}) for $s=t-1$,
$F_j(b_j)>0$ and $F_j(c_j)>0$ for $j=1$, \dots, $t-1$, 
so $I(R''_3)>0$ and $I(R''_4)>0$.  We can therefore remove 
factors of $I(R''_3)^{1/n}$ and $I(R''_4)^{1/n}$ whenever they appear on both 
sides of (\ref{e98}), leaving the equality
$$
\sum_q I(R''_1)^{x/n} I(R''_2)^{(n-x)/n}= I(R''_3)^{x/n} I(R''_4)^{(n-x)/n}.
$$
Here, $x$ is the number of indices $w$ such that $u$ and $w$ are in the
same blocks of $\pi_i$ and $\pi_t$; it follows that there are $n-x$ indices
$w$ such that $u$ and $w$ are in the same block of $\pi_t$ but different
blocks of $\pi_i$.  By construction, $0<x<n$.  Then, by Theorem 11 of
\cite{hwp}, 
\begin{equation}
\label{e100}
\frac
{I(b_1,\ldots,b_{t-1},q,*,\ldots,*)}
{I(b_1,\ldots,b_{t-1},*,*,\ldots,*)}
=
\frac
{I(b'_1,\ldots,b'_{t-1},q,*,\ldots,*)}
{I(b'_1,\ldots,b'_{t-1},*,*,\ldots,*)},
\qquad \hbox{for all $q\in\R$.}
\end{equation}
Since (\ref{e99}) has been proved for $s=t$ and $m=m'-1$, the right-hand
side of (\ref{e100}) must equal $F_t(q)$.  Then from (\ref{e100}) and
(\ref{e99}) for $s=t-1$, it follows that (\ref{e99}) holds for $s=t$
with the given $b_1$, \dots, $b_{t-1}$, and $b_t=q$ arbitrary.
Since $b_1$, \dots, $b_{t-1}$ 
were chosen arbitrarily with $m=m'$, this completes the
secondary induction, and hence the overall induction.

We conclude that for some constant $C>0$ and functions
$F_1,\ldots,F_p:\R\rightarrow\R_{\ge 0}$, we have
$$
I(b_1,\ldots,b_p)
=CF_1(b_1)\cdots F_p(b_p),\qquad \hbox{for all $b_1$, \dots, $b_p\in\R$.}
$$
Therefore, $\cal S$ is a grid, as desired.
\end{proof}

\section{Single-starred motifs}

We now discuss another case where we can determine the asymptotic behavior
of the maximum number of motifs in a set of given size.
Call a motif specification on $p,L$ {\em single-starred} if it contains no
constants, and, for each $i=1$, \dots, $L$, there exists a unique
$j$ such that $\pi_j$ contains the block $\{i\}$.  In this case, any fractional transversal 
$g$ of the hypergraph $H$ of the specification must have $g(i)\ge 1$
for all $i=1$, \dots, $L$; since this condition is also sufficient to
make $g$ a fractional transversal, we see that $\tau^*(H)=L$ and that
an optimal fractional transversal must have $g(i)=1$ for all $i$.

For a single-starred specification, write down the matrix of coordinates 
used in Theorems \ref{thm1} and \ref{thm2}, and make a tableau by
replacing each variable 
corresponding to a block of size 1 with *.  By assumption, there will
be exactly one * in each row.  Each assignment $v$ of values to the remaining 
variables then determines $f_1(v)$, \dots, $f_L(v)$
in $(\R\cup\{*\})^p$, given by 
reading off the rows of the tableau. 
Let $k$ be the number of variables remaining.
By choosing an ordering of the variables, we may think of $v$ as being a 
member of $\R^k$.  

Let $M_1(L):=L(L^L+1)$, and 
let $M(L):= (2 M_1(L))^{L-1} L^2 +1$; we have $M(L)\ge M_1(L)$ for all $L$.
We now make the following definitions:
\begin{itemize}
\item A point $v\in\R^k$ is a {\em center} of a set ${\cal S}\subseteq \R^p$
with size $r$ and
indicator function $I:\R^p\rightarrow\{0,1\}$
if at least $L-1$ of $I(f_1(v))$, \dots, $I(f_L(v))$ are at least $r/M(L)$.
\item A point $v\in\R^k$ is a {\em hypercenter} of a set ${\cal S}\subseteq \R^p$
with size $r$ and indicator function $I:\R^p\rightarrow\{0,1\}$
if all of $I(f_1(v))$, \dots, $I(f_L(v))$ are at least $r/M_1(L)$.
\item A point $w\in\R^p$ is {\em in line with} a point $v\in\R^k$
if its coordinate vector is of the form $f_i(v)$, 
for some $i\in\{1,\ldots,L\}$.
\end{itemize}
It follows that every hypercenter of $\cal S$ is also a center of $\cal S$.

For a single-starred motif specification, let
$C_1$, \dots, $C_q$ be the columns of the tableau where a * occurs,
and let $\alpha_1>0$, \dots, $\alpha_q>0$ be the number of *s occurring
in columns $C_1$, \dots, $C_q$, respectively.
Observe that $q\le L$ and that $\alpha_1+\cdots+\alpha_q=L$.
Set $${\cal C}:=\frac{\alpha_1^{\alpha_1}\cdots \alpha_q^{\alpha_q}}
{L^L}.$$
We have ${\cal C}\ge L^{-L}$.

\begin{lemma} 
\label{lemx}
For each single-starred motif specification on $p,L$, there is some function
$\phi$ such that $\phi(r)\rightarrow 0$ as $r\rightarrow\infty$
and such that, for any $r$, there exists a
set ${\cal S}\subseteq \R^p$ of size $r$ with at least ${\cal C} (1 + \phi(r)) 
r^L$ motifs.
\end{lemma}
\begin{proof}
For any nonnegative integer $N$, we can construct a set $\cal S$ of
size $LN$ by, for each $i=1$, \dots, $q$, taking $\alpha_i N$ points
which have all coordinates 0 except their $C_i$th, which ranges over
$\{1,\ldots, \alpha_i N\}$.  There are then at least
$$(\alpha_1N)^{\alpha_1} \cdots (\alpha_q N)^{\alpha_q}={\cal C} (LN)^L$$
motifs in $\cal S$.  To construct a set of arbitrary size, $r$ say,
we can set $N:=\floor{r/L}$ and pad out $\cal S$ until we get
a set of size $r$ with at least ${\cal C} (L\floor{r/L})^L$ motifs.
Since $(L\floor{r/L})/r\rightarrow 1$ as $r\rightarrow\infty$,
this completes the proof.
\end{proof}

\begin{lemma} 
\label{lemy}
For each single-starred motif specification on $p,L$, there exists
$r_0$ such that whenever $r\ge r_0$ and
${\cal S} \subseteq \R^p$ of size $r$ has the maximum number of motifs
possible for a set of its size, then $\cal S$ has a hypercenter.
\end{lemma}
\begin{proof}
Let $I:\R^p
\rightarrow\{0,1\}$ be the indicator function of $\cal S$.
Then 
\begin{equation}
\label{eee}
{\cal M}({\cal S})=\sum_{v\in\R^k} I(f_1(v))\cdots I(f_L(v)).
\end{equation}
If $\cal S$ has no hypercenter, we know that for each $v$, there
is some $i$ with $I(f_i(v))\le r/M_1(L)$, so (\ref{eee}) can be no more than
$$
\frac{r}{M_1(L)} \sum_{1\le i\le L} \sum_{v\in\R^k} \prod_{j\ne i} I(f_j(v)).
$$
However, each variable must occur at least twice in the tableau, so
for each $i$,
after deleting row $i$, each variable must still occur at least once.
It follows that for each $i$,
\begin{equation}
\label{lemybd}
\sum_{v\in \R^k} \prod_{j\ne i} I(f_j(v))
\le
I(*,\dots,*)^{L-1} = r^{L-1}
\end{equation}
and therefore $\cal S$ has no more than $L r^L/M_1(L)$ motifs.  However,
by Lemma \ref{lemx}, there
exists a set of size $r$ with at least $L^{-L} (1+\phi(r))r^L$ motifs,
and
if we pick $r_0$ large enough, we will have $L^{-L} (1+\phi(r))> (1+L^L)^{-1}$
for all $r\ge r_0$.
This contradicts the assumption that $\cal S$ had the maximum number of motifs.
\end{proof}

\begin{lemma}  
\label{lemz}
For each single-starred motif specification on $p,L$, there exists
$r_0$ such that whenever $r\ge r_0$ and
${\cal S} \subseteq \R^p$ of size $r$ has the maximum number of motifs
possible for a set of its size, then all points in $\cal S$ are in line
with some center.
\end{lemma}
\begin{proof}
If $L=1$, all points in $\R^k$ are centers, so the result is trivial.
We may therefore assume that $L\ge 2$.
Let $I:\R^p \rightarrow\{0,1\}$ be the indicator function of $\cal S$.
If we remove one point, $w$, from ${\cal S}$, then $I(f_i(v))$ does not
decrease if the coordinate vector of $w$ is not of the form $f_i(v)$, 
and decreases by 1 if the coordinate vector is of the form $f_i(v)$.
Therefore, after removing $w$, the number of motifs in $\cal S$ can
decrease by at most
$$
Y:=\sum_{1\le i\le L} \sum_{v\in {\cal T}_i} \prod_{j\ne i} I(f_j(v)),
$$
where ${\cal T}_i$ is the set of points $v\in\R^k$ for which $w$
is of the form $f_i(v)$.
If $w$ is not in line with any center, there are no centers in ${\cal T}_1
\cup\dots\cup{\cal T}_L$,
so for any $i$ and $v\in{\cal T}_i$, 
there must be some $i'\ne i$ with $I(f_{i'}(v))<r/M(L)$.
Therefore, 
$$
Y\le \frac{r}{M(L)} \sum_{1\le i\ne i'\le L} \sum_{v\in{\cal T}_i} \prod_{j\notin\{i,i'\}}
I(f_j(v)).
$$
However, when we sum over $v\in{\cal T}_i$, we are only summing over variables
which do not occur in the $i$th row of the tableau, so after deleting 
the $i$th and $i'$th rows, each variable summed over must still occur at least
once.  Therefore,
$$
\sum_{v\in{\cal T}_i} \prod_{j\notin\{i,i'\}} I(f_j(v))
\le I(*,\dots,*)^{L-2} = r^{L-2}
$$
and so we have decreased the number of motifs in $\cal S$ by at most
$(L^2/M(L)) r^{L-1}$.  Now, if we take $r_0$ 
to be at least as large as in Lemma \ref{lemy}, then $\cal S$ must have
a hypercenter, $v_0$, say, and after deleting $w$, each of 
$I(f_1(v_0))$, \dots, $I(f_L(v_0))$ must be at least $(r/M_1(L))-1$.
If we take $r_0$ to be at least as large as $2 M_1(L)$, 
adding $w$ back to be in line with $v_0$ will then increase the 
number of motifs in $\cal S$ by at least $((r/M_1(L))-1)^{L-1}
\ge (r/(2 M_1(L)))^{L-1}$.
Since $M(L)>(2 M_1(L))^{L-1} L^2$,
this means that moving $w$ to be in 
line with a hypercenter increases the number of motifs in $\cal S$.
This contradicts the assumption that $\cal S$ had the maximum number of motifs.
\end{proof}

\begin{lemma} 
\label{lemctr}
For each single-starred motif specification on $p,L$,
a set $\cal S$ can have at most $L M(L)^{L-1}$ centers.
\end{lemma}
\begin{proof}
For each center $v$, there is some $i$ such that $I(f_j(v))\ge r/M(L)$
for all $j\ne i$.  Fix $i$.  Then, for each $j\ne i$, if we sum $I(f_j(v))$
over the variables occurring in the $j$th row of the tableau,
we will get $r$.  Therefore, there can be at most $M(L)$ assignments of values
to the variables in the $j$th row of the tableau for which $I(f_j(v))\ge r/M(L)$.
Recalling that, after deleting row $i$, each variable must still occur
at least once, it follows that after fixing $i$,
there are at most $M(L)^{L-1}$ possibilities for $v$.
$\cal S$ therefore has at most $L M(L)^{L-1}$ possible centers.
\end{proof}

\begin{theorem}
For each single-starred motif specification on $p,L$,
as $r$ becomes large,
the maximum number of motifs in a set ${\cal S}\subseteq \R^p$
of size $r$ is asymptotic to ${\cal C} r^L$.
Also, there is some $r_0$ such that if $r\ge r_0$, then a set $\cal S$
of size $r$ with a maximum number of motifs for its size
must have exactly one center, which must also be a
hypercenter, and all points in $\cal S$ must be in line with the center.
In addition, the points in $\cal S$ must fall into $q$ lines, $\ell_1$,
\dots, $\ell_q$, say, where each $\ell_i$ is given by fixing
all coordinates except the $C_i$th.
\end{theorem}
\begin{proof}
We have already proved that ${\cal C} r^L$ is an asymptotic lower bound
in Lemma \ref{lemx}.  We now need to prove that it is also an asymptotic
upper bound, and that optimal sets $\cal S$ have the desired form.

Let $r_0$ be at least as large as required in Lemmas \ref{lemy}
and \ref{lemz}, and
let $\cal S$ be a set of size $r\ge r_0$ with a maximum number of motifs
for its size.  
As usual, the number of motifs in $\cal S$ is
$$
\sum_{v\in\R^k} I(f_1(v))\cdots I(f_L(v)).
$$
We may bound this sum above by the sum of $L+1$ pieces:
the first, $Z_1$, will be given by summing over all $v$ where
$f_1(v)\ne f_i(v')$ for all centers $v'$ and $i=1$, \dots, $L$,
\dots, the $L$th, $Z_L$, by summing over all $v$
where $f_L(v)\ne f_i(v')$ for all centers $v'$ and $i=1$, \dots, $L$,
and the last, $Z_{L+1}$, is given by summing over all $v$ where, 
for each $j=1$, \dots, $L$, there is some center $v'$
and $i\in\{1,\ldots,L\}$ for which $f_j(v)=f_i(v')$.  To bound the first
piece, observe that if $f_1(v)\ne f_i(v')$, then there are either no
points whose coordinate vectors are both of the form $f_1(v)$ and $f_i(v')$
(in the case where the * is in the same position in $f_1(v)$ and $f_i(v')$)
or at most one point (in the case where the * is in different positions
in $f_1(v)$ and in $f_i(v')$.)  Also, by Lemma \ref{lemz}, every point in 
$\cal S$ is of the form $f_i(v')$, for some $i$ and some center $v'$.
Then, summing over $i$ and $v'$ and using Lemma \ref{lemctr}, we see that $I(f_1(v))\le L^2 M(L)^{L-1}$,
so by (\ref{lemybd}), the first piece of the sum is no more than
$L^2 M(L)^{L-1} r^{L-1}$.  We can bound the 2nd through $L$th pieces in the
same way,
so 
\begin{equation}
\label{e101a}
{\cal M}({\cal S})\le Z_1+\cdots+Z_{L+1}\le Z_{L+1}+L^3 M(L)^{L-1} r^{L-1}.
\end{equation}

For $j=1$, \dots, $L$, let $\sigma_j\in\{1,\ldots,q\}$ 
be such that there is a * in row $j$ and column $C_{\sigma_j}$ of the 
tableau, and for 
each $i=1$, \dots, $q$, let ${\cal S}_i$ be the set of points of
$\cal S$ whose coordinate vectors are of the form $f_j(v)$, for some
center $v$ of $\cal S$, and some $j$ with $\sigma_j=i$.
For $i=1$, \dots, $q$,
we let ${\cal S}_i$ have indicator function $I_i$ and size $r_i$.
By Lemma \ref{lemz}, all points in $\cal S$
must be in line with some center, so ${\cal S}=\cup_{1\le i\le q} {\cal S}_i$.
Now, if there is a point $w$ which is in ${\cal S}_i \cap
{\cal S}_{i'}$, for some $i\ne i'$, then $w$ is of the form $f_j(v)$ and also of the form
$f_{j'}(v')$, for some centers $v$ and $v'$ of $\cal S$.  Every 
coordinate of $w$ except the $C_i$th is determined by $v$ and $j$, and
every coordinate of $w$ except the $C_{i'}$th is determined by $v'$ and $j'$.
Then, since $i\ne i'$, $w$ is determined by $v$, $v'$, $j$, and $j'$.  
Since $\cal S$ can have at most $L M(L)^{L-1}$ centers, there are at 
most $L^4 M(L)^{2L-2}$ possibilities for $w$.  Therefore, 
\begin{equation}
\label{e101b}
r_1+\cdots+r_q\le r + L^5 M(L)^{2L-2}.
\end{equation}

However, by the definition of $Z_{L+1}$,
\begin{equation}
\label{e102a}
Z_{L+1}\le 
\sum_{v \in \R^k} I_{\sigma_1}(f_1(v)) \cdots
I_{\sigma_L}(f_L(v)),
\end{equation}
and we have the obvious inequality
\begin{equation}
\label{e102b}
\sum_{v \in \R^k} I_{\sigma_1}(f_1(v)) \cdots
I_{\sigma_L}(f_L(v))
\le 
\left(\sum_{v_1} I_{\sigma_1}(v_1)\right)
\cdots
\left(\sum_{v_L} I_{\sigma_L}(v_L)\right),
\end{equation}
where, for each $i$, $v_i$ is summed over all elements of $(\R\cup\{*\})^p$
which contain exactly one *, in coordinate $C_{\sigma_i}$.
However, the right-hand side of this inequality is just
$$
\prod_{1\le i\le L} r_{\sigma_i} = \prod_{1\le j\le q} r_j^{\alpha_j}
$$
and therefore from (\ref{e102a}) and (\ref{e102b}) we get
\begin{equation}
\label{e101c}
Z_{L+1} \le \prod_{1\le j\le q} \alpha_j^{\alpha_j}
\prod_{1\le j\le q} (r_j/\alpha_j)^{\alpha_j}.
\end{equation}
The asymptotic upper bound ${\cal C} r^L$ now follows immediately from
(\ref{e101a}), (\ref{e101b}),  (\ref{e101c}),
and Theorem 9 of \cite{hwp}.

By Lemma \ref{lemy}, we already know that $\cal S$ has a hypercenter,
$v_0$, say, and by Lemma \ref{lemz}, we know that all its points are in line 
with some center.
It remains to show that $\cal S$ has only one center, and that its 
points fall into $q$ lines as claimed.  However, if $\cal S$ had a center,
$v'$, say, not equal to $v_0$, then $v'$ and $v_0$ would have to
assign some variable
a different value.  Since each variable occurs in at least two rows of the
matrix of coordinates, there must therefore be some $i$ for which 
$I(f_i(v'))\ge r/M(L)$ and 
$f_i(v')\ne f_i(v_0)$.  Now, the term
$$
I_{\sigma_1}(f_1(v_0)) \cdots I_{\sigma_{i-1}}(f_{i-1}(v_0)) I_{\sigma_i}(f_i(v')) I_{\sigma_{i+1}}(f_{i+1}(v_0)) \cdots I_{\sigma_L}(f_L(v_0))
$$
will occur on the right-hand side but not the left-hand side of (\ref{e102b}),
and it will have 
value at least $(r/M(L))(r/M_1(L))^{L-1}$; this means that the left- and right-hand
sides of (\ref{e101c}) must differ by at least $r^L/(M(L)M_1(L)^{L-1})$.  By 
our asymptotic estimate, this will prevent $\cal S$ from containing the 
maximum possible number of motifs, if we take $r_0$ sufficiently large.  We 
conclude that $\cal S$ must have only one center, $v_0$.  Finally, if there 
are $i$ and $j$ for which $\sigma_i=\sigma_j$ but $f_i(v_0)\ne f_j(v_0)$, the 
term
$$
I_{\sigma_1}(f_1(v_0)) \cdots I_{\sigma_{i-1}}(f_{i-1}(v_0)) I_{\sigma_i}(f_j(v_0)) I_{\sigma_{i+1}}(f_{i+1}(v_0)) \cdots I_{\sigma_L}(f_L(v_0))
$$
will occur on the right-hand side but  not the left-hand side of (\ref{e102b}),
and it will have value at least $(r/M_1(L))^L$.  As before, this 
will prevent $\cal S$
from containing the maximum possible number of motifs.  Therefore, we conclude
that the set $\{f_1(v_0), \ldots, f_L(v_0)\}$ has only $q$ distinct elements,
one with a * in position $C_i$ for $i=1$, \dots, $q$.  Then, 
since every point in $\cal S$ must be in line with the unique center,
$v_0$, every point in $\cal S$ must have coordinate vector matching
one of these $q$ distinct elements.  This establishes the only remaining
claim.
\end{proof}

\section{Conclusion}

To finish, we ask two questions:
\begin{enumerate}
\item For a general motif specification, what is the asymptotic behavior of the maximum number of motifs in a set of given size?
\item For motif specifications with no constants and uniform hypergraph, sets with the maximum number of motifs for their size are always grids.  Also, for single-starred motif specifications, large enough sets with the maximum possible number of motifs for their size are always the union of $q$ lines.  Are large sets with the maximum possible number of motifs for their size always unions of a small number of grids?
\end{enumerate}

\vspace{.75in}

\end{document}